\documentclass[12pt,hyperref]{article}
\usepackage{graphicx} 
\usepackage{graphicx}
\usepackage{subcaption}
\usepackage{float} 
\usepackage{amssymb}
\usepackage{amsmath}
\usepackage{mathtools}
\usepackage[thmmarks,amsmath]{ntheorem} 
\usepackage{lineno}
\usepackage{xcolor}
\definecolor{darkblue}{RGB}{0, 0, 100}
\definecolor{darkgreen}{RGB}{0, 100, 0}

{\raise0.5ex\hbox{$#1$}\! \left/ \! \lower-0.5ex\hbox{$#2$}\right.}

\usepackage{algorithm}
\usepackage{algpseudocode}
\usepackage{amsmath}

\usepackage{bm} 
\usepackage{extarrows}
\usepackage{mathrsfs} 

\usepackage{algorithm,algpseudocode,float}
\usepackage{lipsum}


\usepackage{booktabs} 
{
    \theoremstyle{nonumberplain}
    \theoremheaderfont{\bfseries}
    \theorembodyfont{\normalfont}
    \theoremsymbol{\mbox{$\Box$}}
    \newtheorem{proof}{Proof}
}

\usepackage{theorem}
\newtheorem{theorem}{Theorem}[section]
\newtheorem{proposition}{Proposition}[section]
\newtheorem{lemma}{Lemma}[section]
\newtheorem{definition}{Definition}[section]

\newtheorem{corollary}{Corollary}[section]

\newtheorem{claim}{Claim}[section]

\newtheorem{conjecture}{Conjecture}[section]
{
    \theoremheaderfont{\bfseries}
    \theorembodyfont{\normalfont}
}
\newcommand{\RNum}[1]{\uppercase\expandafter{\romannumeral #1\relax}}
\allowdisplaybreaks[4]

\graphicspath{{figure/}}                                 
\usepackage[a4paper]{geometry}                        
\begin{document}
\title{\bf On the $k$-linkage  problem for generalizations of semicomplete digraphs \footnote{The author's work is supported by National Natural Science Foundation of China (No.12071260).}}
\date{}
\author{Jia Zhou\textsuperscript{1}, J{\o}rgen Bang-Jensen\textsuperscript{2}, Jin Yan\textsuperscript{1} \footnote{Corresponding author. E-mail address: yanj@sdu.edu.cn.} \unskip\\[2mm]
\textsuperscript{1} School of Mathematics, Shandong University, Jinan 250100, China
\\
\textsuperscript{2} Department of Mathematics and Computer Science, University of\\ Southern Denmark, Odense DK-5230, Denmark}

\maketitle

\begin{abstract}
A directed graph (digraph) $ D $ is \textbf{$ k $-linked} if $ |D| \geq 2k $, and for any $ 2k $ distinct vertices $ x_1, \ldots, x_k, y_1, \ldots, y_k $ of $ D $, there exist vertex-disjoint paths $ P_1, \ldots, P_k $ such that $ P_i $ is a path from $ x_i $ to $ y_i $ for each $ i \in [k] $. In 1980, Thomassen conjectured that there exists a function $ f(k) $ such that every $ f(k) $-strong digraph is $ k $-linked. He later disproved this conjecture by showing that $ f(2) $ does not exist for general digraphs and proved that the function $f(k)$ exists for the class of tournaments. In this paper we consider  a large class $\mathcal{ D} $ of digraphs which  includes all \textbf{semicomplete digraphs} (digraphs with no pair of non-adjacent vertices) and all \textbf{quasi-transitive digraphs} (a digraph $D$ is quasi-transitive if for any three vertices $x, y, z$ of $D$, whenever $xy$ and $yz$ are arcs, then $x$ and $z$ are adjacent). We prove that every $ 3k $-strong digraph $D\in \mathcal{D}$ with minimum out-degree at least $ 23k $ is $ k $-linked. A digraph $D$ is \textbf{$l$-quasi-transitive} if whenever there is a path of length $l$ between vertices $u$ and $v$ in  $D$ the vertices  $u$ and $v$ are adjacent. Hence 2-quasi-transitive digraphs are exactly the quasi-transitive digraphs. We prove that there is a function $f(k,l)$ so that every $f(k,l)$-strong $l$-quasi-transitive digraph is $k$-linked. The main new tool in our proofs significantly strengthens an important property of vertices with maximum in-degree in a tournament. While Landau in 1953 already proved that such a vertex $v$ is reachable by all other vertices by paths of length at most 2, we show that, in fact, the structure of these paths is much richer. In general there are many such paths for almost all out-neighbours of $v$ and this property is crucial in our proofs.
\end{abstract}

\vspace{1ex}
{\noindent\small{\bf Keywords: } Connectivity; Linkage; Semicomplete composition }
\vspace{1ex}

{\noindent\small{\bf AMS subject classifications.} 05C20, 05C38, 05C40}

\section{Introduction}

Connectivity is among the most fundamental concepts in graph theory. As such, this research area is of great significance. Throughout this paper, we let $k \geq 2$ be an integer, and abbreviate "vertex-disjoint" to "disjoint" for brevity. A (di)graph is (\textbf{strongly}) \textbf{$k$-connected} (briefly "$k$-strong" for digraphs) if it has at least $k+1$ vertices and, after deleting any set $S$ of at most $k-1$ vertices, there is a path from $x$ to $y$ for any pair of distinct vertices $x$ and $y$. A (di)graph $D$ is \textbf{$k$-linked} if for any $2k$ distinct vertices $x_1, \ldots, x_k, y_1, \ldots, y_k$ in $D$, there exist $k$ disjoint paths $P_1, \ldots, P_k$ such that $P_i$ starts at $x_i$ and terminates at $y_i$ for each $i \in [k]$. It is easy to verify that every $k$-linked graph is $k$-connected, but the converse does not hold  not even for planar graphs as there exist 5-connected planar graphs which are not 2-linked (see e.g. the graph $G4$ in \cite[Figure 10.10]{bang2009}).
Thomas and Wollan \cite{Thomas(2005)} have shown that a connectivity of $10k$ is sufficient to ensure that a graph is $k$-linked. In 1980, two different authors \cite{seymourDM29,thomassenEJC1} proved simultaneously that every 6-connected graph is 2-linked. (This inspired Thomassen to conjecture that a similar result should hold for digraphs).

\begin{conjecture}\label{conj1}
  \cite{thomassenEJC1} There exists a function $f(k)$ such that every $f(k)$-strong digraph is $k$-linked.
\end{conjecture}

In 1991 Thomassen \cite{Thomassen(1991)} disproved Conjecture \ref{conj1} by constructing, for each  $l \in \mathbb{N}$, an  $l$-strong digraph that is not 2-linked. However, the situation is entirely different for some special classes of digraphs. In 1984, Thomassen \cite{Thomassen(1984)} verified that there exists some constant $C$ such that $f(k) \leq Ck!$ for tournaments. A digraph $D$ is a \textbf{tournament} if it is obtained by orienting each edge of a complete graph in exactly one direction. K\"{u}hn, Lapinskas, Osthus, and Patel \cite{Khn(2014)} improved this bound to $f(k) \leq 10^4 k \log k$ for tournaments, and they conjectured that $f(k)$ is linear in $k$. Subsequently, in 2015, this conjecture was confirmed by Pokrovskiy \cite{Pokrovskiy(2015)} who proved that $f(k) \leq 452k$ for tournaments. In 2021, Meng, Rolek, Wang, and Yu \cite{Meng(2021)} improved this connectivity bound and presented that $f(k)\leq 40k-31$ for tournaments. Further, Bang-Jensen and Johansen \cite{Bang(2021)} verified that every $(13k- 6)$-strong tournament with minimum out-degree at least $28k- 13$ is $k$-linked. There have been further improvements to this result, as detailed in \cite{ Girao(2021),Snyder(2019), Zhou, Zhou(2024)}. Additionally, Pokrovskiy \cite{Pokrovskiy(2015)} proposed the following conjecture. The \textbf{minimum semi-degree} of a digraph $D$ is the minimum over all out-degrees and in-degrees of vertices of $D$.

\begin{conjecture}\label{1.2}
  \cite{Pokrovskiy(2015)}  There exists an integer $g(k)$ such that every $2k$-strong tournament with minimum semi-degree at least $g(k)$ is $k$-linked.
\end{conjecture}

 Gir\~{a}o, Popielarz, and Snyder \cite{Girao(2021)} constructed an infinite family of $2.5k$-strong tournaments with small semi-degree that are not $k$-linked. This shows that large minimum out-degree is necessary when the connectivity is at most $2.5k$. Furthermore, they showed in \cite{Girao(2021)} that every $(2k + 1)$-strong tournament with minimum out-degree at least $Ck^{31}$ is $k$-linked for some constant $C$. In \cite{Zhou(2024)} the first and third author of this paper disproved Conjecture \ref{1.2} by constructing, for each $n \geq 14k^2$, a family of $2k$-strong tournaments of order $n$ with minimum semi-degree at least $(n - 2k - 2)/(2k + 2)$ that are not $k$-linked. This shows that $f(k)\geq 2k+1$ for the class of tournaments. It was also shown in \cite{Zhou(2024)} that there exists a constant $C$ such that every $(2k + 1)$-strong semicomplete digraph with minimum out-degree at least $Ck^4$ is $k$-linked. A digraph $D$ is \textbf{semicomplete} if there is at least one arc between any two vertices $x$ and $y$. Hence this class of digraphs contains all tournaments. 

 In an effort to identify large classes of digraphs for which Conjecture \ref{conj1} holds, we study so-called compositions of (semicomplete) digraphs.  Let $H$ be a digraph on $h \geq 2$ vertices $v_1, \ldots, v_h$, and let $S_1, \ldots, S_h$ be a collection of {disjoint} (but possibly isomorphic) digraphs. The \textbf{composition} $D = H[S_1, \ldots, S_h]$ is a digraph with vertex set $V(D) = V(S_1) \cup \cdots \cup V(S_h)$ and arc set $A(D) = \left(\bigcup_{i=1}^h A(S_i)\right) \cup \left\{s_i s_j : s_i \in V(S_i), s_j \in V(S_j), v_i v_j \in A(H)\right\}$. We say that $ S_i $ is a \textbf{part} of $ D $. Specially, if $H$ is semicomplete, then $D = H[S_1, \ldots, S_h]$ is called a \textbf{semicomplete composition}. For more results on semicomplete compositions, see \cite{bang2018,bangJGT95,bangEJC120}. Let $\delta^+(D)$ denote the \textbf{minimum out-degree} of $D$. We present the main results in this paper below.

\begin{theorem}\label{the0}
Every $3k$-strong semicomplete digraph $D$ with $\delta^+(D)\geq 22 k$ is $k$-linked.
\end{theorem}

If we let each $S_i$ be a vertex, then the semicomplete composition $D = H[S_1, \ldots, S_h]$ becomes a semicomplete digraph. Using this fact and Theorem \ref{the0} we can prove the following result.

 \begin{theorem}\label{the1}
 Suppose $D=H[S_1,\ldots,S_h] $ is a semicompete composition with the property that  $|V(D)\setminus S_i|\geq 2k-3$ for each $i\in [h]$. If $D$ is $3k$-strong with $\delta^+(D)\geq 23 k$, then $D$ is $k$-linked.
\end{theorem}

 The restriction on the size of each $S_i$ in Theorem \ref{the1} is necessary (see Proposition \ref{prop2} in Section 6).

A digraph $D$ is called an \textbf{extended semicomplete digraph} if $D = H[S_1, \ldots, S_h]$ for some semicomplete digraph $H$ and arcless digraphs $S_1, \ldots, S_h$. Clearly, extended semicomplete digraphs are a special class of semicomplete compositions. Moreover, if the extended semicomplete digraph $D = H[S_1, \ldots, S_h]$ is $3k$-strong, then $|D \setminus S_i| \geq 3k$ for each $i \in [h]$. Hence, we get Corollary \ref{cor1}.

\begin{corollary}\label{cor1}
Every $3k$-strong extended semicomplete digraph $D$ with $\delta^+(D)\geq 23 k$ is $k$-linked.
\end{corollary}

A digraph $D$ is \textbf{quasi-transitive} if for any three vertices $x, y, z$ of $D$, such that $xy$ and $yz$ are arcs of $D$, there is at least one arc between $x$ and $z$. Theorem \ref{the2} below provides a characterization of strong quasi-transitive digraphs. This implies, in particular, that strong quasi-transitive digraphs are semicomplete compositions.
\begin{theorem}\label{the2}
\cite{Bang(1995)} Let $D$ be a strong quasi-transitive digraph. Then there exists a strong semicomplete digraph $H$ and quasi-transitive digraphs $S_1, \ldots, S_h$ such that $D = H[S_1, \ldots, S_h]$, where $h = |V(H)|$ and each $S_i$ is either a vertex or non-strong.
\end{theorem}

It follows from Theorem \ref{the2} that a $3k$-strong quasi-transitive digraph $D = H[S_1, \ldots, S_h]$ is also a semicomplete composition with $|D \setminus S_i| \geq 3k$ for each $i \in [h]$. Consequently, we obtain Corollary \ref{cor2}, which significantly improves the result of Bang-Jensen \cite{Bang(1999)}, who demonstrated that $f(k) \leq 2^k \cdot k (k-1)!$ for quasi-transitive digraphs.

\begin{corollary}\label{cor2}
Every $3k$-strong quasi-transitive digraph $D$ with $\delta^+(D)\geq 23 k$ is $k$-linked.
\end{corollary}

A digraph $D$ is \textbf{$l$-quasi-transitive} if for every pair of vertices $u, v$ of $D$, the existence of a $(u, v)$-path of length $l$ in $D$ implies that there is at least one arc between $u$ and $v$. Hence, the 1-quasi-transitive digraphs are the  semicomplete digraphs and the 2-quasi-transitive digraphs are exactly the  quasi-transitive digraphs.
 For a number of results on quasi-transitive and general $l$-quasi-transitive digraphs, see Chapter 8 in
\cite{bang2018}. As $ l $ increases, the condition forcing adjacencies becomes (much) more relaxed, making the structure of $ l $-quasi-transitive digraphs more difficult to work with than that of quasi-transitive digraphs. Still  we are able to prove that Conjecture \ref{conj1} holds for  $l$-quasi-transitive digraphs.

\begin{theorem}\label{fina}
Let $l$ be an integer with $l\geq 2$. Every $ 81k^2(l+2)^2$-strong $l$-quasi-transitive digraph is $k$-linked.
\end{theorem}

The rest of the paper is structured as follows. Section 2 introduces key concepts and lemmas that will be utilized throughout the article. In Section 3, we will first provide the proof of Theorem \ref{the0}. In the proof of Theorem \ref{the0}, we provide a new linking strategy. This involves constructing a so-called nearly in-dominating  set (see Definition \ref{def1}), applying Menger's Theorem, and skillfully using Lemma \ref{lem2.2} to progressively construct the  desired set of disjoint paths. Section 4 proves Theorem \ref{the1} using a subtle construction. Section 5 presents the proof of Theorem \ref{fina}. Here, we solve the linking problem for an $l$-quasi-transitive digraph $D$ by constructing an associated semicomplete digraph $D'$, finding appropriate disjoint paths in $D'$ and then transforming these back to the desired linking in $D$. Finally, Section 6 gives some remarks and open problems.

\section{Terminology and Preliminaries}

\subsection{Notation}
Notation not introduced here is consistent with \cite{bang2009}. Let $\mathbb{N}$ be the set of natural numbers. For an integer $i$, we use the notation  $[i] = \{1, \ldots, i\}$, and $[i, i+j] = \{i, i+1, \ldots, i+j\}$. The digraphs considered in this paper are always simple, i.e., without loops and multiple arcs. Let $D$ be a digraph with vertex set $V(D)$ and arc set $A(D)$. We use $|D|$ to represent the number of vertices in $D$. For two vertices $x, y \in V(D)$, we denote the arc from $x$ to $y$ as $xy$. For a vertex $x$ of a digraph $D$, we define $N^+_D(x) = \{y \mid xy \in A(D)\}$ (resp. $N^-_D(x) = \{y \mid yx \in A(D)\}$) as the \textbf{out-neighborhood} (resp. \textbf{in-neighborhood}) of $x$, and $d^+_D(x) = |N^+_D(x)|$ (resp. $d^-_D(x) = |N^-_D(x)|$) as the \textbf{out-degree} (resp. \textbf{in-degree}) of $x$. Furthermore, let $\delta^+(D)$ (resp. $\delta^-(D)$) denote the minimum out-degree (resp. minimum in-degree) of $D$. 

For a set $X \subseteq V(D)$, the subgraph of $D$ induced by $X$ is denoted by $D\langle X \rangle$, and the digraph obtained from $D$ by deleting $X$ and all arcs incident with $X$ is denoted by $D \setminus X$. For a digraph $D$ and sets $A, B \subseteq V(D)$, we use the notation $A \rightarrow B$ to indicate that every arc of $D$ between $A$ and $B$ is directed from $A$ to $B$. In particular, we say that $x$ \textbf{dominates} $y$ if $x \rightarrow y$. For a digraph $ D $, we denote by $\kappa(D) $ the degree of strong connectivity of $ D $. So $\kappa(D)=k$ means that $D$ is $k$-strong but not $(k+1)$-strong.

A \textbf{matching} in a digraph $D=(V,A)$ is a set of arcs $A'=\{u_1v_1,u_2v_2,\ldots{},u_rv_r\}\subset A$  such that no two arcs in $A'$ have a common vertex. We say that $A'$ is a matching from $X=\{u_1,u_2,\ldots{},u_r\}$ to $Y=\{v_1,v_2,\ldots{},v_r\}$. Suppose below that $P = x_1x_2 \cdots x_t$ is a directed path of $D$. Then we say that  $x_1$ (resp. $x_t$) is  the \textbf{initial} (resp. \textbf{terminal}) vertex of $P$ and say that $P$ is an $(x_1,x_t)$-path. The vertices $x_2\ldots{},x_{t-1}$ are the \textbf{interior} vertices of $P$. The \textbf{length} of $P$ is the number of its arcs, and we denote a path of length $l$ as an $l$-path. Furthermore, if $P$ is as above and $z$ is not a vertex of $P$ such that  $zx_1$ is an arc of $D$, then $zx_1P$ represents a new path from $z$ to $x_1$ and then along $P$ to $x_t$. We  denote by  \textbf{$P[x, y]$} the subpath of $P$ from $x$ to $y$. Furthermore, $P$ is \textbf{minimal} if, for every $(x_1, x_t)$-path $Q$, either $V(P) = V(Q)$ or $V(Q)$ is not properly contained in $V(P)$. Two paths are \textbf{independent} if neither contains an interior vertex of the other. Given a collection of paths $\mathcal{Q}$, we use $\text{Ini} (\mathcal{Q})$ to represent the set of all initial vertices of paths in $\mathcal{Q}$, $\text{Ter} (\mathcal{Q})$ to represent the set of all terminal vertices of paths in $\mathcal{Q}$, and $\mathbf{\text{Int}(\mathcal{Q})} = V(\mathcal{Q}) \setminus (\text{Ini} (\mathcal{Q}) \cup \text{Ter} (\mathcal{Q}))$.

\subsection{Key Definitions and Lemmas}
In our proofs, we utilize the following easy consequence of Menger's Theorem.
\begin{theorem}\label{menger}
\cite{menger} Let $D$ be a $k$-strong digraph. Then for any two disjoint sets $\{x_1, \ldots, x_k\}$ and $\{y_1, \ldots, y_k\}$ of $V(D)$, there exist disjoint $(x_i, y_{\pi(i)})$-paths for some permutation $\pi$ of $\{1, 2, \ldots, k\}$.
\end{theorem}


\begin{definition}\label{def1}
 Let $c \in \mathbb{N}$, and let $u, v$ be two vertices of a digraph $D$. We say $v$ is \textbf{$\boldsymbol{c}$-good for $\boldsymbol{u}$} in $D$ if either $v$ dominates $u$ or there exist at least $c$ independent $(v, u)$-paths of length 2 in $D$.
\end{definition}

\begin{definition}\label{def2}
 A vertex $u$ is a \textbf{nearly in-dominating  vertex} of $D$ if, for every $c \in \mathbb{N}$, all but at most $2c$ vertices are $c$-good for $u$.\\
  We call a set $U$ of vertices  a \textbf{nearly in-dominating  set} of $D$ if, for every vertex $u\in U$ and every $c \in \mathbb{N}$, all but at most $2c$ vertices in $D \setminus U$ are $c$-good for $u$ in $D$.
\end{definition}

The following lemma, which plays a crucial role in our proofs, establishes the existence of a nearly in-dominating  vertex in a semicomplete digraph.
\begin{lemma}\label{lemma8}
Every semicomplete digraph $D$ contains a nearly in-dominating  vertex $u$.
\end{lemma}

\begin{proof}
Fix $c \in \mathbb{N}$, and let $T$ be any spanning tournament of $D$. We select $u$ to be a vertex of maximum in-degree in $T$  and proceed to prove that $u$ is a nearly in-dominating  vertex of $D$. Let
\[
S = \{w \in N^+_T(u) \mid w \text{ is not $c$-good for $u$ in } D\}.
\]
We may assume that $|S| > 2c$ as otherwise the claim would already be established. By the definition  of $S$, for every vertex $w \in S$,  we have $u \rightarrow w$ and $T$ contains no collection of $c$ independent $(w, u)$-paths of length 2. Moreover, by
the choice of $u$, we have  $d^-_T(w) \leq d^-_T(u)$.

We claim that, except for up to $c$ vertices, every vertex in $T$ is either an in-neighbour of $u$ or an out-neighbour of $w$. This follows from the following calculation:
\begin{equation*}
\begin{aligned}
|N^+_{T}(w) \cup N^-_{T}(u)|
&= |N^+_{T}(w)|+| N^-_{T}(u)|-|N^+_{T}(w) \cap N^-_{T}(u)|\\
&\geq |T|-1-d^-_{T}(w)+d^-_{T}(u)-(c-1)\\
&\geq |T|-1-c+1\\
&\geq |T|-c.
\end{aligned}
\end{equation*}
It follows from $u \rightarrow S$ that $d^+_{S}(w) \geq |S| - c$. Since $w$ is an arbitrary vertex of $S$ we have $\delta^+(T\langle S\rangle) \geq |S| - c$. However, this is impossible because
\begin{equation*}
\begin{aligned}
|S| &\geq 2\delta^+(T\langle S\rangle) + 1 \geq 2|S| - 2c + 1,
\end{aligned}
\end{equation*}
which implies that $2c \geq |S| + 1 > 2c + 1$. This completes the proof.
\end{proof}



An \textbf{in-king} in a tournament $T$ is a vertex $v$ such that for every other vertex $u\in V(T)$, there is a path from $u$ to $v$ of length at most 2. Landau \cite{Landau(1953)} proved that the vertex with the maximum in-degree in every tournament is an in-king. Actually, Lemma \ref{lemma8} shows that this in-king is also a nearly in-dominating vertex, thereby {significantly strengthening the properties of those in-kings which are also vertices of maximum in-degree}.

\begin{definition}\label{def3}
Let $\gamma$ be an integer, and $D$ be a digraph with a vertex $v$ and a subset $U$ of $V(D)\setminus \{u\}$, we say that the vertex $v$ is a:
\begin{itemize}
    \item \textbf{$\boldsymbol{\mathbf{\gamma}}$-out-dominator of $\boldsymbol{U}$} if $v$ has at least $\gamma$ out-neighbours in $U$; and
    \item \textbf{$\boldsymbol{\gamma}$-in-dominator of $\boldsymbol{U}$} if $v$ has at least $\gamma$ in-neighbours in $U$.
\end{itemize}
\end{definition}

 The following property, which has been used in many recent papers on linkages, was introduced in \cite{Pokrovskiy(2015)}. Let $ D $ be a digraph, and let $ X = \{x_1, \ldots, x_k\} $ and $ Y = \{y_1, \ldots, y_k\} $ be disjoint subsets of $ V(D) $. We say $ \boldsymbol{X} $ \textbf{anchors} $ \boldsymbol{Y} $ if there exist $ k $ disjoint $ (x_i, y_{\pi(i)}) $-paths for every permutation $ \pi $ of $ [k] $. Furthermore, if these $ k $ paths have lengths at most 3, we say that $ \boldsymbol{X} $ \textbf{short anchors}  $ \boldsymbol{Y }$.

 The following lemma plays a central role in the proof of Theorem \ref{the0} and will be used repeatedly. It establishes that, under specific structural conditions, a subset can short anchor a specific subset while avoiding another subset $W$ within a digraph. Recall that for a collection $\mathcal{Q}$ of paths, $\text{Ini} (\mathcal{Q})$ denotes the set of all initial vertices of the paths in $\mathcal{Q}$.

\begin{lemma}\label{lem2.2}
Let $D$ be a digraph with three disjoint subsets $X$, $Y$ and $W$ of $V(D)$ such that $|X| = |Y| = k$. Suppose that $U$ is a nearly in-dominating set of $D\setminus (X\cup Y)$ of order $3k$, and $\mathcal{Q}$ is a collection of $k$ disjoint paths from $U$ to $Y$ with the minimum number of vertices among all such collections. If there exists a set $A\subseteq V(D)\setminus (W\cup \text{Ini} (\mathcal{Q}))$ of order at most $k$ such that each vertex $a\in A$ has at least $7k+3|W|+7|A|$ out-neighbors in $ D \setminus (X \cup Y \cup U) $ each of which is a 1-in-dominator of $ U\setminus \text{Ini}  (\mathcal{Q})$,  then $A$ short anchors every subset $S\subseteq \text{Ini}  (\mathcal{Q})\setminus W$ for which $|S|=|A|$ in the subdigraph induced by the vertices in $ (V(D)\setminus (V(\mathcal{Q})\cup W \cup X))\cup  (S\cup A)$.
\end{lemma}


\begin{proof}
  Let $ A = \{a_1, \ldots, a_{|A|}\} $ and $ S = \{q_1, \ldots, q_{|A|}\} $. We need to find $|A|$ disjoint paths $P_1, \ldots, P_{|A|}$ of length at most 3, where each path $P_i$ is an $(a_i, q_i)$-path for $i \in [|A|]$ in the subdigraph induced by the vertices in $ (V(D)\setminus (V(\mathcal{Q})\cup W \cup X))\cup  (S\cup A)$. 

To construct paths $P_1, \ldots, P_{|A|}$, we first choose a set $\{a^+_i, i\in [|A|]\}$ of $|A|$ distinct vertices such that for each $i\in [|A|]$,
\begin{itemize}
\item[(i)] $a_i\rightarrow a^{+}_i$ and $a^{+}_i \in V(D) \setminus (X \cup Y \cup U\cup Y_2 \cup W)$;
\item[(ii)] $a^{+}_i$ is a 1-in-dominator of  $ U\setminus \text{Ini}  (\mathcal{Q})$;
\item[(iii)] $a^{+}_i$ is $(3k + |W| + 3|A|)$-good for $q_i$ in $D\setminus (X\cup Y)$.
\end{itemize}
By the hypothesis of the lemma, for each vertex $a_i\in A$, among the out-neighbors of vertex $a_i\in A$ in $D \setminus (X \cup Y \cup U \cup Y_2 \cup W)$, the number of 1-in-dominators of $U \setminus \text{Ini}(\mathcal{Q})$ is at least $7k + 3|W| + 7|A| - k - |W| = 6k + 2|W| + 7|A|$.  Recall that $ U $ is a nearly in-dominating  set of $ D \setminus (X \cup Y) $ and $q_i\in U$ for all $i\in [|A|]$.
Taking $ c = 3k + |W| + 3|A| $, it follows from Definition \ref{def2} that we can find $6k + 2|W| + 7|A|-2c= |A| $ vertices satisfying (i)-(iii) for each $a_i$. Hence such set $\{a^+_i, i\in [|A|]\}$ exists. For all $i \in [|A|]$, if $a^{+}_i \rightarrow q_i$, then we set $P_i = a_i a^{+}_i q_i$; otherwise, since $a^{+}_i$ is $(3k + |W| + 3|A|)$-good for $q_i$ in $D \setminus (X \cup Y)$, the number of 2-paths from $a^{+}_i$ to $q_i$ that are internally disjoint from $X \cup Y \cup A \cup \text{Ini}(\mathcal{Q}) \cup \bigcup_{i=2}^3 Y_i \cup W \cup \{a^{+}_i , i \in [|A|]\}$ is at least
\[
3k + |W| + 3|A| - |A| - 3k - |W| - |A| \geq |A|,
\]
which implies that there exist $|A|$ distinct vertices $a^{++}_i \notin X \cup Y \cup A \cup \text{Ini}(\mathcal{Q}) \cup \bigcup_{i=2}^3 Y_i \cup W \cup \{a^{+}_i : i \in [|A|]\}$ such that $a^{+}_i a^{++}_i q_i$ is a 2-path. Let $P_i = a_i a^{+}_i a^{++}_i q_i$. Thus we have shown that we can greedily find $|A|$ disjoint paths $P_1, \ldots, P_{|A|}$, each of length at most 3, such that each path $P_i$ is an $(a_i, q_i)$-path and has the form $P_i = a_i a^{+}_i q_i$ or $P_i = a_i a^{+}_i a^{++}_i q_i$.

Next, we show that the collection of paths $\mathcal{P} = \{P_i \mid i \in [|A|]\}$ we constructed in $D \setminus W$ is internally disjoint from $V(\mathcal{Q}) \cup X$. From our construction, it is clear that $V(\mathcal{P}) \cap W = \emptyset$, and $\mathcal{P}$ is internally disjoint from $X \cup \text{Ini}(\mathcal{Q})$. It remains to show that $V(\mathcal{P}) \cap V(\mathcal{Q} \setminus S) = \emptyset$. Assume for contradiction that there exists a path $ P \in \mathcal{P} $ and a path $ Q \in \mathcal{Q} $ such that $ P $ intersects $ Q $ at a vertex in $ V(Q) \setminus S $. By construction, $ P $ is either of the form $ P = a_i a^+_i q_i $ or $ P = a_i a^+_i a^{++}_i q_i $. Because (ii), the vertex $a^+_i$ has an in-neighbour $z\in U \setminus \text{Ini}(\mathcal{Q})$. If $a^+_i \in V(Q)$, then (i) implies that the subpath $Q[q, a^+_i]$ has at least 3 vertices, where $q$ is the initial vertex of $Q$. Replacing the subpath $Q[q, a^+_i]$ by the shorter path $za^+_i$ contradicts the minimality of $|V(\mathcal{Q})|$. Hence, $a^+_i \notin V(\mathcal{Q})$.
This shows that we must have $P=a_ia^+_ia^{++}_iq_i$ and $a^{++}_i \in V(Q)$. Since $a^{++}_i \notin \text{Ini}(\mathcal{Q}) \cup \bigcup_{j=2}^3 Y_j$, the path $za^+_ia^{++}_i$ is shorter than $Q[q, a^{++}_i]$. Replacing the subpath $Q[q, a^{++}_i]$ by the shorter path $za^+_ia^{++}_i$ also contradicts the minimality of $|V(\mathcal{Q})|$. Therefore,  $V(\mathcal{P}) \cap V(\mathcal{Q} \setminus S) = \emptyset$ and the proof is complete.
\end{proof}

The proof of Theorem \ref{the1} requires the following lemma.
\begin{lemma}\label{lem2.3}
\cite{Bang(1999)} Let $D = H[S_1,\ldots,S_h]$ be a composition, where $H$ is a strong digraph on $h\geq2$ vertices. Let $D_0 = H[S^0_{1},\ldots,S^0_{h}]$ be the digraph obtained from $D$ by deleting every arc which lies inside $S_i$, for all $i\in [h]$. If $H$ has at least three vertices, then $D$ is $k$-strong if and only if $D_0$ is $k$-strong.
\end{lemma}



\section{Proof of Theorem \ref{the0}}
In this section, we let $D$ be a $3k$-strong semicomplete digraph with $\delta^+(D)\geq 22 k$, and assume $X=\{x_1,\ldots,x_k\}$ and $Y=\{y_1,\ldots,y_k\}$ are disjoint subsets of $V(D)$. We first provide a sketch of the proof for Theorem \ref{the0}.

\textbf{Sketch of proof of Theorem \ref{the0}.}
Construct a nearly in-dominating  set $U=\{u_i \mid i\in [3k]\}$ of $D\setminus (X\cup Y)$ by iteratively applying Lemma $\ref{lemma8}$ to $D\setminus (X\cup Y \cup \bigcup_{j<i} \{u_j\})$. We  {then} partition the set $X$ into $X_1$ and $X_2$ such that  {a matching from $X_1$ to a set $X^*_1$ can be easily found}. By Menger's Theorem, we  {can find}  a collection $\mathcal{Q}$ of $k$ disjoint paths from $U$ to $Y$ that minimizes $|V(\mathcal{Q})|$ while avoiding the vertices in  {$X_1\cup X^*_1$}.  {Using the matching from $X_1$ to  {$X^*_1$,} a collection $\mathcal{P}_1$ of disjoint 2-paths from $X_1$ to $U\setminus \text{Ini}(\mathcal{Q})$ that avoids   {all}  vertices in $V(\mathcal{Q})$ can be easily found}. By applying Lemma $\ref{lem2.2}$, we can link $X_2$ to the corresponding set $S\subseteq \text{Ini}(\mathcal{Q})$ in $(D\setminus (V(\mathcal{P}_1)\cup V(\mathcal{Q})))\cup S$. Finally, using Lemma $\ref{lem2.2}$ again, we find disjoint paths to link $\text{Ter}(\mathcal{P}_1)$ with $\text{Ini}(\mathcal{Q})\setminus S$.  {Thus, we complete the proof (see Fig. \ref{fig6}).}
\hfill $\square$

 \begin{figure}[htbp]
	\centering
	\begin{minipage}{0.8\linewidth}
		\centering
		\includegraphics[width=0.8\linewidth]{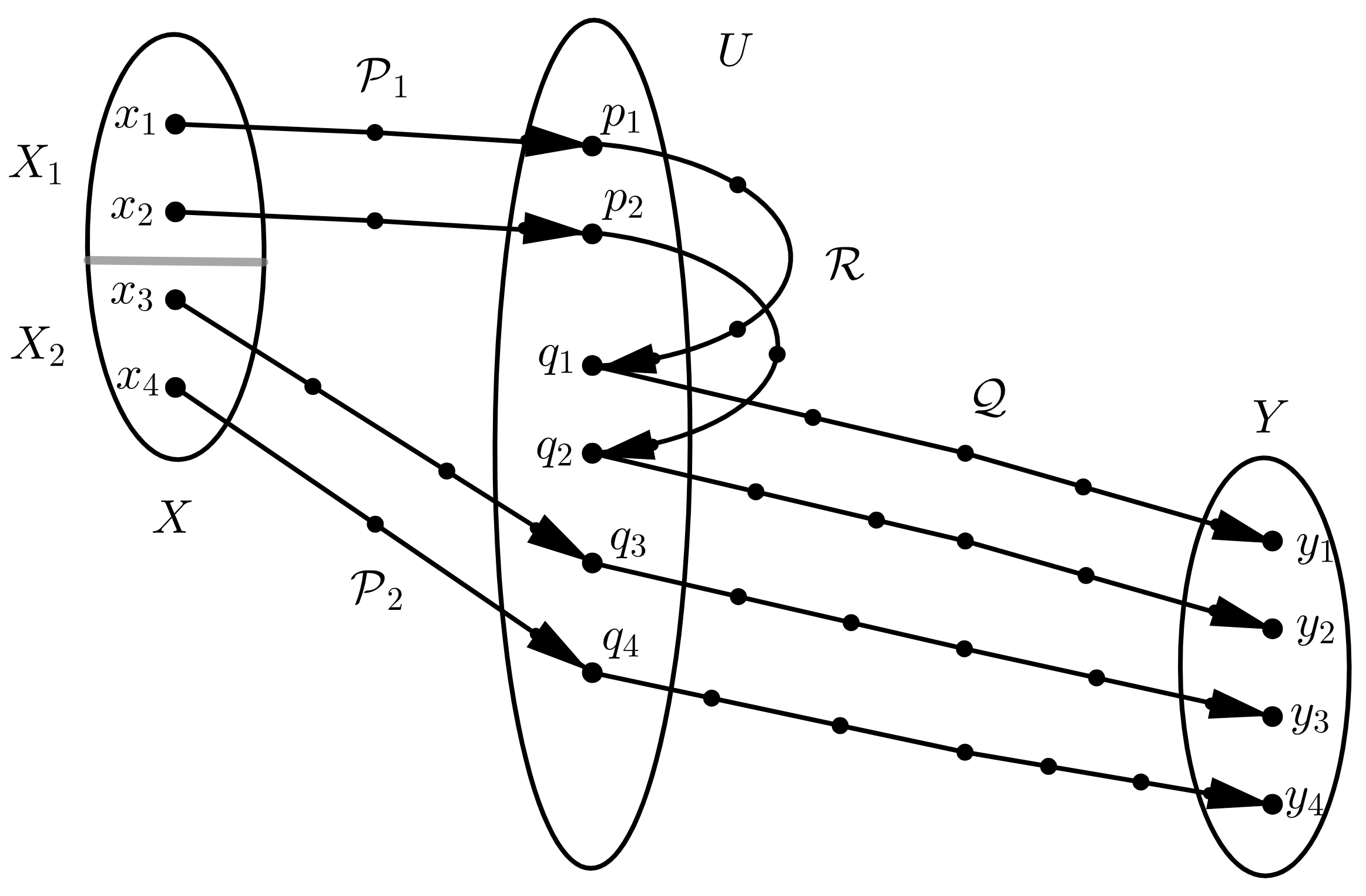}
		\caption{The basic structure of the linking paths for $k=4$.}\label{Fig1}
		\label{fig6}
	\end{minipage}
\end{figure}


\textbf{Proof of Theorem \ref{the0}.}
By iteratively applying Lemma \ref{lemma8} with $D := D \setminus (X \cup Y \cup \bigcup_{j < i} \{u_j\})$, we obtain a set  {$U = \{u_i \mid i \in [3k]\}$ such that   $u_i$
  is a nearly in-dominating  vertex of $D \setminus (X \cup Y \cup \bigcup_{j < i} \{u_j\})$}.
We divide $X$ into two parts: $X_1$ and $X_2$. Set $X_1 = \{x_i \in X \mid N^+_D(x_i) \setminus (X \cup Y \cup U) \text{ contains at least }  {k}  {\text{ distinct } 2k\text{-out-dominators of } U}\}$, and $X_2 = X \setminus X_1$. Without loss of generality, we can assume that $X_1 = \{x_1, \ldots, x_l\}$ and $X_2 = \{x_{l+1}, \ldots, x_k\}$. It follows from the definition of $X_1$ that we can greedily construct a matching $\{x_i x^+_i \mid i \in [l]\}$ such that each $x^+_i\in N^+_D(x_i) \setminus (X \cup Y \cup U) $ is a $2k$-out-dominator of $U$.

Since $D$ is $3k$-strong, by Menger's Theorem, we can find $k$ disjoint paths from $U$ to $Y$ which avoid all vertices in $X \cup \{x^+_1, \ldots, x^+_l\}$. Choose a collection $\mathcal{Q}$ of such paths that minimizes $|V(\mathcal{Q})|$. This implies that only the initial vertices in $\mathcal{Q}$ belong to $U$. Denote by $Y_i$ the set of all $i$-th vertices on paths in $\mathcal{Q}$ for  {$i \geq 2$} and for each $i\in [k]$, let $q_i$ be the initial vertex of the path from $\mathcal{Q}$ which ends in $y_i$. 

We now aim to construct pairwise disjoint paths linking $x_i$ to $q_i$ for $i \in [k]$. Since each $ x^+_i $ is a $ 2k $-out-dominator of $ U $ and  {$|\text{Ini}(\mathcal{Q})|= k $} for all $ i \in [l] $, we can obtain a collection $ \mathcal{P}_1 $ of $ l $ disjoint 2-paths from $ X_1 $ to $ U \setminus \text{Ini}(\mathcal{Q}) $, such that these 2-paths are internally disjoint from $ V(\mathcal{Q}) $. We then show that there exists a collection  {$\mathcal{P}_2=\{P_{l+1},\ldots{},P_k\}$ such that $P_j$ is an $(x_j,q_j)$-path of length at most 3 which avoids all vertices of $V(\mathcal{P}_1)\cup V(\mathcal{Q})$ for each  {$j\in [l+1,k]$.}}

 {Note that for every vertex $v\notin U$, if $v$ does not a $2k$-out-dominator of $U$, then $v$ is a $(k + 1)$-in-dominator of $U$ since $|U| = 3k$.} By the  {definition} of $X_2$,  {for each $i \in [k] \setminus [l]$, the number of $(k+1)$-in-dominators of} $U$ in $N^+_D(x_i) \setminus (X \cup Y \cup U)$ is at least
\[
\delta^+(D) - |X \cup Y \cup U| - (k-1) \geq 16k > 7k + 3l + 7(k - l).
\]
 {Note that each 
such vertex is a 1-in-dominator of $U \setminus \text{Ini}(\mathcal{Q})$}. Hence, the desired collection $\mathcal{P}_2$ can be found by applying Lemma \ref{lem2.2} with $A = X_2$,  {$S=\{q_{l+1},\ldots,q_k\}$} and $W = \text{Int}(\mathcal{P}_1)$.

Let $p_i$ be the terminal vertex of the path starting at $x_i$ in $\mathcal{P}_1$. The proof is completed by finding  {a set $\mathcal{R}$ of} $l$ disjoint paths linking $p_i$ with $q_i$ for $i \in [l]$ in $(D \setminus  {(}V(\mathcal{Q}) \cup V(\mathcal{P}_1) \cup V(\mathcal{P}_2) {)} )\cup \{p_1, \ldots, p_l, q_{1}, \ldots, q_l\}$. We achieve this purpose by using Lemma \ref{lem2.2} once again. Notice that $ |\text{Int}(\mathcal{P}_1 \cup \mathcal{P}_2) |\leq 2k-l$.  {Observe that} each vertex in $\{p_1, \ldots, p_l\}$ has at least $\delta^+(D) - |X \cup Y \cup U| \geq 17k \geq 7k + 3(2k - l) + 7l$ out-neighbours in  {$D\setminus (X \cup Y \cup U)$}, and each out-neighbour of $p_i$ is a 1-in-dominator of $U \setminus \text{Ini}(\mathcal{Q})$ by definition. Hence, we can apply Lemma \ref{lem2.2} with $A = \{p_1, \ldots, p_l\}$,  {$S=\{q_1,\ldots,q_l\}$} and $W = \text{Int}(\mathcal{P}_1 \cup \mathcal{P}_2)$ to find the desired paths.

By combining the corresponding paths in $\mathcal{P}_1$ and $\mathcal{P}_2$, we can find $k$ disjoint $(x_i, y_i)$-paths for $i \in [k]$, as desired  {(see Fig. \ref{Fig1})}.
\hfill $\square$

\section{Proof of Theorem \ref{the1}}
 {Let }$D = H[S_1, \ldots, S_h]$  {be} a semicomplete composition  {satisfying the conditions in Theorem \ref{the1}.} 
 We proceed by induction on $k$. The base case $k = 1$ is trivial. Assume that $k \geq 2$ and  {that} the theorem holds for $k - 1$. We now show that it holds for $k$. Suppose that $D$ is not $k$-linked, that is,  {there exists a pair of disjoint vertex sets $X = \{x_1, \ldots, x_k\}$ and $Y = \{y_1, \ldots, y_k\}$ such that $D$ has  no collection of  $k$ disjoint paths linking $x_i$ to $y_i$ for $i \in [k]$}.  {Suppose first that there is an index $i\in [k]$ such that $x_i$ dominates $y_i$. }Since $D \setminus \{x_i, y_i\}$ is a $(3k - 2)$-strong semicomplete composition with $| { {D}} \setminus ( {S_j} \cup \{x_i, y_i\})| \geq 2(k - 1) - 3$ for all  {$j \in [h]$} and $\delta^+(D \setminus \{x_i, y_i\}) \geq  {23k} - 2$,  {it follows from the induction assumption that} the digraph $ D \setminus \{x_i, y_i\}$ is $(k-1)$-linked. {This contradicts the assumption that $D$ has no linking from $X$ to $Y$.  Hence, we may assume that}
\begin{equation}\label{4}
  { D \text{ does not contain any of the arcs } x_1y_1,\ldots{},x_ky_k}
\end{equation}

\begin{claim}\label{4.1}
 {$h\geq 3$}.
\end{claim}
\begin{proof}
 {Otherwise, $D=H[S_1,S_2]$. Since $D$ is strong,  {it follows from the definition of a composition that $D$ contains the  2-cycle $xyx$ for  every  vertex $x$ of $S_1$ and every vertex $y$ of $S_2$}.  {Using this we will find an $(x_i,y_i)$-path of length 2 for some $i\in[k]$:}  The inequality (\ref{4}) indicates that for all $i \in [k]$,  {the vertices $x_i$ and $y_i$ are either both located in $S_1$ or both in $S_2$}. Since $|D|\geq 3k$, either $S_1$ or $S_2$ must have a vertex $v \notin X \cup Y$.  {Assume first that $v\in S_2$}. If  {$x_i, y_i \in S_1$, for some $i\in [k]$,} then $x_i v y_i$ is an $(x_i, y_i)$-path in $D$. Otherwise, $ x_i, y_i \in S_2 $ for all $ i \in [k] $. Since $ h \geq 2 $, there must be a vertex $ u \in S_1 \setminus (X \cup Y) $ such that $ x_i u y_i $ is an $ (x_i, y_i) $-path for all $ i \in [k] $. In either case, we obtain an $(x_i, y_i)$-path $P$ of length 2 that uses at most two vertices from  {$V(S_1)$ and $V(S_2)$}, for some $i\in [k]$. Consequently, $D \setminus V(P)$ is a $3(k-1)$-strong semicomplete composition, and $|D \setminus (S_i \cup V(P))| \geq 2(k-1) - 3$ for all $i \in [2]$. Moreover, $\delta^+(D \setminus V(P)) \geq  {23k} - 3$. By the induction hypothesis, $D \setminus V(P)$ is $(k-1)$-linked. This contradicts our assumption that $ D $ is not $ k $-linked.}
\end{proof}

 {Let $D_0=H[S^0_1,\ldots,S^0_h] $ be the digraph obtained from $D$ by deleting all arcs which lies inside $S_i$ for all $i\in [h]$. By Lemma \ref{lem2.3}, $D_0$ is $3k$-strong. We now construct a new semicomplete digraph $D'$ with $V(D')=V(D_0)$ from $D_0$ as follows. For each part $S^0_i$ in $D_0$, let $Y'_i=Y\cap V(S^0_i)$. We then construct $S'_i$ by adding new arcs in $S^0_i$ such that both $ Y'_i $ and $V(S_i)\setminus Y'_i$ induce complete digraphs, and $ Y'_i $ dominates all other vertices in  $ V(S_i)\setminus Y'_i $.}

 Clearly,  {the digraph} $D' = H[S'_1, \ldots, S'_h]$ is a  {$3k$-strong} semicomplete digraph.  {Note that, for each $i\in [h]$, every vertex in $Y'_i$ has the out-degree at least $23k$, and every vertex in $V(S_i)\setminus Y'_i$ has out-degree at least $22k$ as we maybe loose arcs to $Y'_i$ in $S'_i$. Hence, we obtain that }$\delta^+(D') \geq  {22k}$. Also, by our construction, we get that
\begin{equation}\label{3}
 {\text{if $x_j, y_j \in V(S'_i)$, then there is no $(x_j, y_j)$-path in $(S'_i\setminus (X\cup Y))\cup \{x_j,y_j\}$ }}
\end{equation}


Theorem \ref{the0} implies that $D'$ is $k$-linked. Consequently, $D'$ contains $k$ disjoint minimal paths $P_1, \ldots, P_k$, where $P_i$ is an $(x_i, y_i)$-path. Observe that if there exists a pair of vertices $x_i, y_i$ lying in the same part $S'_j$, there is no $(x_i, y_i)$-path in $S'_j$ by (\ref{3}). This implies that all paths $P_1, \ldots, P_k$ only use the arcs between different parts,  {provided that each path $P_i$ is chosen to be a minimal path.} Thus, these $k$ paths do not use the newly added arcs in $D'$, so they are also present in $D$. This contradicts our assumption that such paths do not exist in $D$. The proof is complete.
 \hfill $\square$



 \section{Proof of Theorem \ref{fina}}

 The following lemma is an important tool in our proof. It was shown initially for $n\geq 11k$  in \cite{Pokrovskiy(2015)}, and further improvements have been made in \cite{Meng(2021)}.
\begin{lemma}\label{anchor}
\cite{Meng(2021)} Let $T$ be a tournament on $n$ vertices. If $n \geq 9k - 6$, then there are two disjoint vertex sets, each of size $k$, such that one short anchors the other.
\end{lemma}

 {For a given $81k^2(l+2)^2$-strong $l$-quasi-transitive digraph $D$ and disjoint vertex sets $X$ and $Y$ of size $k$ we prove the result by constructing a semicomplete digraph $D'$. Then we use Lemma \ref{anchor} and Lemma \ref{lemma8} to find a linking from $X$ to $Y$ in $D'$ such that these paths  can be transformed back to the desired linking in $D$.} To construct a  {suitable}  semicomplete digraph from an $l$-quasi-transitive digraph, we need the following lemma. Here, the \textbf{distance} $d(x, y)$ from a vertex $x$ to a vertex $y$ is the length of a shortest $(x, y)$-path.
\begin{lemma}\label{key}
\cite{C(2012)} Let $l$ be an integer with $l \geq 2$. Suppose $D$ is an $l$-quasi-transitive digraph and contains a $(u, v)$-path. If $d(u, v) \geq l$, then $d(v, u) \leq l + 1$.
\end{lemma}

\begin{proof}\textbf{of Theorem \ref{fina}.}
Let $D$ be a $81k^2(l+2)^2$-strong $l$-quasi-transitive digraph, and let $X=\{x_1,\ldots,x_k\}$ and $Y=\{y_1,\ldots,y_k\}$ be two  {disjoint} subsets of $V(D)$. Our objective is to construct pairwise disjoint $(x_i, y_i)$-paths for $i \in [k]$. We define $D_0 = D \setminus (X \cup Y)$. Note that the class of $ l $-quasi-transitive digraphs is closed under induced subdigraphs. 

We construct an auxiliary semicomplete digraph $D'$ with vertex set $V(D)$ as follows.  {Let $f(k,l)=\binom{9k-6}{2}(l+2) + (2l+5)k+9k$.} For any two non-adjacent vertices $u$ and $v$ of  {$V(D_0)$}, we add the arc $uv$ to $D'$ if and only if there are at least $ {f(k,l)}$ independent $(u, v)$-paths each of length at most $l+1$ in $D$. We refer to arcs added in this manner as \textbf{new arcs}.  {The new arcs} will be used to ensure the semicompleteness of $D'$. Additionally, we add the arc set $A = \{ux_i, i\in [k] \mid x_i \in X, u \in V(D) \setminus \{x_i\}, \text{ and } ux_i \notin A(D)\} \cup \{y_iv, i\in [k] \mid y_i \in Y, v \in V(D) \setminus \{y_i\}, \text{ and } y_iv \notin A(D)\} $ to $D'$ (see Fig. \ref{fig7}). Note  {that no linking from $X$ to $Y$ will use   any arc from  $A$.}

 \begin{figure}[htbp]
	\centering
	\begin{minipage}{0.8\linewidth}
		\centering
		\includegraphics[width=0.8\linewidth]{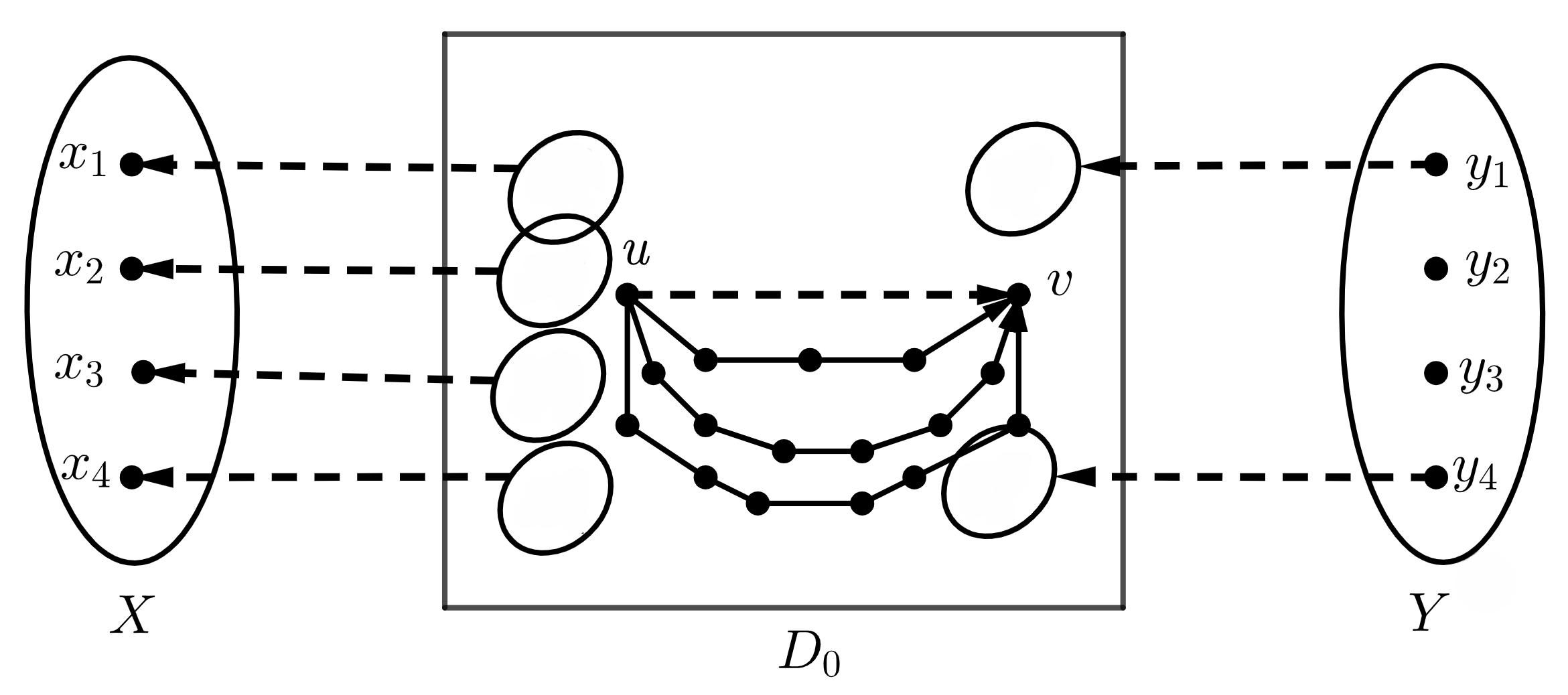}
		\caption{The construction of $D'$: The dashed lines indicate that these arcs are absent in $D$. We add arcs in $D'$ in the direction indicated by the arrows. Specifically, an arc $uv$ is added to $D'_0$ if and only if there exist $ {f(k,l)}$ independent  $(u,v)$-paths in $D_0$.}
		\label{fig7}
	\end{minipage}
\end{figure}

To prove that $D'$ is a semicomplete digraph, we only need to examine non-adjacent vertices within $V(D_0)$. By Lemma \ref{key} and the strong connectivity of $D_0$, for any two non-adjacent vertices $u$ and $v$ in $D_0$, there exists a $(u,v)$-path or a $(v,u)$-path of length at most $l + 1$. We proceed as follows: Suppose we have $ i $ independent paths $ Q_1, Q_2, \ldots, Q_i $ in $ D_0 $, each either from $ u $ to $ v $ or from $ v $ to $ u $  {and all of} length of at most $ l + 1 $. If the remaining digraph $(D_0 \setminus \bigcup_{j \in [i]} V(Q_j)) \cup \{u, v\}$  {is} strong, we can obtain another $(u,v)$-path or $(v,u)$-path $Q_{i+1}$ of length at most $l+1$ in $(D_0 \setminus \bigcup_{j \in [i]} V(Q_j)) \cup \{u, v\}$ by applying Lemma \ref{key} again.  {Since a path of length at most $l+1$ has at most $l+2$ vertices {and
\begin{align*}
&\kappa(D)-2k - 2f(k,l)(l + 2)\\
\geq & 81k^{2}(l+2)^{2}-2k-2\left(\frac{81k^2-117k+42}{2}(l+2)+ (2l+5)k+9k\right)(l+2)\\
= & 81k^2(l^2+4l+4)-2k-\left((81k^{2}-117k+42)(l+2)+4lk+28k\right)(l + 2)\\
=& 81k^2l^2+324k^2l+324k^2-2k-(81k^2l-113kl+42l+162k^2-206k+84)(l+2)\\
=&113kl^2+432kl+410k-42l^2-168l-168>0\  (\text{as }k\geq 1, l\geq 2).
\end{align*}}}
we can recursively derive  {$2f(k,l)$} independent paths of length at most $l+1$ in $D_0$,  {where each path is either directed from $u$ to $v$ or from $v$ to $u$.} By the pigeonhole principle, at least ${f(k,l)}$ of these paths between $u$ and $v$ are in the same direction. This means that $ D' $ is a semicomplete digraph by the construction of $ D' $.  {Furthermore, for each pair of non-adjacent vertices $u,v$ of $D_0$ such that $ D_0 $ contains at least $ {f(k,l)} $ independent $ (u,v) $-paths of lengths at most $ l+1 $ we refer to  {the collection of} independent $ (u, v) $-paths with a length of at most $ l+1 $ as the \textbf{available paths} for the arc $ uv \in A(D') $}.

Let $D_0' = D' \setminus (X \cup Y)$. For all $i \in [9k - 6]$, {set $D_i' = D_{i-1}' \setminus \{u_i\}$, and we} choose $u_i$ to be a nearly in-dominating  vertex in $V(D_{i-1}')$ using Lemma \ref{lemma8}. Let $U = \{u_i \mid i \in [9k - 6]\}$. Let $T_U'$ be a spanning tournament of $D' \langle U \rangle$.  {As $|U| = 9k - 6$, we can apply Lemma \ref{anchor} to the tournament $T_U'$ and we obtain two disjoint} sets: $U_1 = \{u_{\alpha_1}, u_{\alpha_2}, \ldots, u_{\alpha_k}\}$ and $U_2 = \{u_{\beta_1}, u_{\beta_2}, \ldots, u_{\beta_k}\}$ such that $U_i\subseteq U, \ i\in [2]$ and  $U_1$ anchors $U_2$ in $T_U'$.

To link $X$ to $U_1$ in $D$, we consider each $i \in [k]$. Since
\[
d_{D \setminus (U \cup X \cup Y)}^{+}(x_i) \geq d_D^{+}(x_i) - |X \cup Y \cup U| \geq \kappa (D) - 11k + 6 > 23k,
\]
we can greedily find $X^+ = \{x_1^{+}, x_2^{+}, \ldots, x_k^{+}\}$ such that $x_i \rightarrow x_i^{+}$ in $D$ and $x_i^+$ is $11k$-good for $u_{\alpha_i}$ in $D_0'$, for all $i \in [k]$. This is possible since $u_{\alpha_i} \in U$ is a nearly in-dominating  vertex  {(recall Definition \ref{def2} with $c=11k$)}. Furthermore, if $x_i^+\rightarrow u_{\alpha_i}$ in $D'$, we let $P_i = x_i x_i^+ u_{\alpha_i}$. If not, we observe that the number of independent $(x_i^+, u_{\alpha_i})$-paths of length 2 that are internally disjoint from 
 {$X^+\cup U$} is at least $11k - |X^+ \cup U| \geq k$ in $D'$. This implies that $D'$ contains $k$ disjoint paths $P_1', \ldots, P_k'$, each of length at most 3 and linking $x_i$ to $u_{\alpha_i}$.

For each path $P_i' = x_i x_i^+ x_i^{++} u_{\alpha_i}$, one or both of the arcs $x_i^+ x_i^{++}$ and $x_i^{++} u_{\alpha_i}$ may be new arcs. Without loss of generality, we assume they are all new arcs. By the construction of the new arcs, there are at least {$f(k,l)$} available paths in $D_0$ for each new arc $x_i^+ x_i^{++}$ or $x_i^{++} u_{\alpha_i}$.  {Furthermore}, the vertices in $\bigcup_{i=1}^k {\text{Int}(P_i')} \cup U$ may intersect at most $11k-6$ of these paths. Thus,  {by selecting appropriate disjoint available paths in $D_0$ to replace the new arcs in each $P_i'$ we can obtain
disjoint paths $P_1'',\ldots{},P''_k$ such that $P''_i$ and an ($x_i,u_{\alpha_i})$-path for  $i \in [k]$}. This leads to $ {|V(P''_i)|}\leq 2(l+1)+3=2l+5$, and then
\[
\left| \bigcup_{i=1}^k V(P_{i}'') \right| \leq (2l + 5)k.
\]

Next, we proceed to find $k$ disjoint paths from $U_2$ to $Y$ in $D$. The spanning tournament $T_U'$ of $D' \langle U \rangle$ has $\binom{9k - 6}{2}$ arcs. By the construction of $D'$, for any two non-adjacent vertices $u, v$ in $D \langle U \rangle$, there are at least $f(k,l)-(2l + 5)k={\binom{9k-6}{2}(l+2)+9k}$ available paths in $D_0$ for $u, v$ that avoid the set $\bigcup_{i=1}^k  {V(P_i'')}$. Clearly, replacing a new arc in $D' \langle U \rangle$ with an available path that avoids $\bigcup_{i=1}^k  {V(P_i'')}$ ensures that the available paths corresponding to different new arcs remain independent {and all available paths are internally disjoint from $U$}. Define $U^*$ as the union of all vertices from  {the available paths that were used in this replacement}.  Thus, $U^* \cup U$ contains at most $\binom{9k - 6}{2}(l + 2)$ vertices, and $U^*\cap \bigcup_{i=1}^k V(P_i'')=\emptyset$. Let $B = \left(\bigcup_{i=1}^k  {V(P_i'')}\right) \cup U^* \cup U$, and observe that
\begin{align*}
|B|&\leq (2l + 5)k+\binom{9k - 6}{2}(l + 2)\\&=2lk+5k+\frac{(81k^2-117k+ 42)(l + 2)}{2}\\
&=\frac{81k^2l - 113kl}{2}+81k^2-112k+21l+42 <81k^2(l+2)^2 \leq \kappa (D).
\end{align*}
Further, $\kappa (D)-|B|+k\geq k$.
{Hence} the subdigraph $(D \backslash B) \cup U_2$ is $k$-strong. Thus,  {by Theorem \ref{menger},} there exists a permutation $\pi$ of $\{1, \ldots, k\}$ such that there are $k$ disjoint paths $R_1^\pi, \ldots, R_k^\pi$ in $(D \backslash B) \cup U_2$ such that $R_i^\pi $ is a $(u_{\beta_{\pi(i)}},y_i)$-path.

Finally, we link $U_1$ to $U_2$ in $D$. For each $i \in [k]$, we find paths $Q_i^\pi$ to connect $P_i''$ with $R_i^\pi$, forming the desired path $P_i$ from $x_i$ to $y_i$. Recall that $U_1$  {short} anchors $U_2$ in $T_U'$. Thus, there exist disjoint $(u_{\alpha_i}, u_{\beta_{\pi(i)}})$-paths $Q_i'$, $i \in [k]$, of lengths at most 3 in $D' \langle U \rangle$. By substituting  {each} new arc in $D' \langle U \rangle$ with an appropriate available path in $D \langle U^* \rangle$, we obtain $k$ disjoint paths $Q_i^\pi$ in $D \langle U \cup U^* \rangle$, where each $Q_i^\pi$ is a $(u_{\alpha_i}, u_{\beta_{\pi(i)}})$-path. Finally, let
\[
P_i = P_i''[x_i, u_{\alpha_i}] \cup Q_i^\pi[u_{\alpha_i}, u_{\beta_{\pi(i)}}] \cup R_i^\pi[u_{\beta_{\pi(i)}}, y_i] \quad \text{for} \quad i \in [k].
\]
Thus, $P_1, P_2, \ldots, P_k$ are the desired paths. This completes our proof of the theorem.
\end{proof}

\section{Remarks}

As mentioned in the introduction, Thomassen proved the following.

\begin{proposition}\label{prop1}
\cite{Thomassen(1991)} For each $l \in \mathbb{N}$, there exists an $l$-strong digraph $D^*$ that is not 2-linked.
\end{proposition}

To illustrate that the restriction placed on $S_i$ in Theorem \ref{the1} is necessary, we introduce the following proposition.

\begin{proposition}\label{prop2}
Let $l$ be an integer. There exists an $l$-strong semicomplete composition $D = R[S_1, \ldots, S_r]$ that is not $k$-linked, where $S_1, \ldots, S_r$ are disjoint digraphs satisfying $|D \setminus S_r| \leq 2k - 4$.
\end{proposition}

\begin{proof}
Let $r = 2k - 3$ and let $R$ be a tournament with vertex set $\{v_i \mid i \in [r]\}$ such that $R\langle \{v_i \mid i \in [r-1]\}\rangle$ is a transitive tournament with arc set $\{v_i v_j \mid i < j\}$, and $\{v_i v_r, v_r v_i \mid i < r\}$. For each $i \in [r-1]$, let $S_i$ be a single vertex, and let $S_r$ be the $l$-strong digraph $D^*$ as stated in Proposition \ref{prop1}. Our construction clearly shows that $D = R[S_1, \ldots, S_r]$ is $l$-strong. We now claim that $D$ is not $k$-linked. Since $S_r$ is not 2-linked, there exist two pairs $\{u,x\}$ and $\{v,y\}$ such that $S_r$ contains no disjoint $(u, v)$-path {and} $(x, y)$-path. Now, define $x_i = v_{2i}$ and $y_i = v_{2i-1}$ for each $i \in [k-2]$. Also set $x_{k-1} = u$, $x_k = x$, $y_{k-1} = v$, and $y_k = y$. It is easy to check that $D$ contains no $k$ disjoint paths linking $x_i$ to $y_i$ for $i \in [k]$.
\end{proof}

{In fact, Lemma \ref{anchor} was further refined into the following lemma in \cite{chen}, which in turn  {will lead} to a slight improvement of Theorem \ref{fina}.}
\begin{theorem}
{\cite{chen}  Let $T$ be a tournament on $n$ vertices. If $n\geq 8.5k-6$, then there are two disjoint vertex sets each of size $k$ such that one short anchors another.}
\end{theorem}
{However, in  {order to keep our  calculations reasonably simple we have decided  to use } Lemma \ref{anchor} to prove Theorem \ref{fina}.}

The most obvious open problem is to reduce the connectivity bound from $3k$ to $2k+1$ in Theorem \ref{the0}. Indeed, Gir\~{a}o, Snyder, and Popielarz \cite{Girao(2021)} proposed the following conjecture.

\begin{conjecture}
\cite{Girao(2021)} There exists a constant $C > 0$ such that every $(2k+1)$-connected tournament with minimum out-degree at least $Ck$ is $k$-linked.
\end{conjecture}

It is interesting to explore both the relationship between connectivity and out-degree and the classes of digraphs that satisfy Conjecture \ref{conj1}. It would also be interesting to identify a larger class of digraphs, or a class not covered by Theorem \ref{the1}, that are $k$-linked under the same connectivity assumptions.




\end{document}